\documentclass[11pt]{amsart}

\usepackage{amsmath}
\usepackage{amssymb}
\usepackage{amsthm}
\usepackage{tikz}
\usepackage{tikz,pgfplots}
\usetikzlibrary{decorations.markings}
\usepackage{comment}

\usepackage{graphicx,tikz}
\usetikzlibrary{shadows}
\usepackage{subcaption}
\usepackage{todonotes}
\usepackage{graphicx}
\usetikzlibrary{arrows}
\usepackage[utf8]{inputenc}
\usepackage[english]{babel}
\usepackage{array}
\usepackage{booktabs}
\usepackage{mdframed}

\usepackage{enumitem}
\usepackage{xcolor}

\newcommand{\ex}{{\rm ex}}

\usepackage[margin=2.5cm]{geometry}

\pgfplotsset{compat=1.17}

\newtheorem{theorem}{Theorem}

\newtheorem{lemma}[theorem]{Lemma}
\newtheorem{conjecture}[theorem]{Conjecture}

\newtheorem{remark}[theorem]{Remark}

\theoremstyle{definition}
\newtheorem{definition}[theorem]{Definition}

\tikzset{middlearrow/.style={
		decoration={markings,
			mark= at position 0.5 with {\arrow[scale=2]{#1}} ,
		},
		postaction={decorate}
	}
}

\tikzset{midarrow/.style={
		decoration={markings,
			mark= at position 0.3 with {\arrow[scale=1.5]{#1}} ,
		},
		postaction={decorate}
	}
}

\begin{document}
	
	\title{Tur\'{a}n problems for $k$-geodetic digraphs}
	
	\author{James Tuite}
	\address{School of Mathematics and Statistics, Open University, Walton Hall, Milton Keynes, UK}
	\email{james.tuite@open.ac.uk}
	
	\author{Grahame Erskine}	
	\thanks{Corresponding author: Grahame Erskine}
	\address{School of Mathematics and Statistics, Open University, Walton Hall, Milton Keynes, UK}
	\email{grahame.erskine@open.ac.uk}
	
	\author{Nika Salia}
	\address{Alfr\'ed R\'enyi Institute of Mathematics, Budapest, Hungary}
	\email{nikasalia@yahoo.com}

	\begin{abstract}
		A digraph $G$ is \emph{$k$-geodetic} if for any pair of (not necessarily distinct) vertices $u,v \in V(G)$ there is at most one walk of length $\leq k$ from $u$ to $v$ in $G$. In this paper we determine the largest possible size of a $k$-geodetic digraph with given order. We then consider the more difficult problem of the largest size of a strongly-connected $k$-geodetic digraph with given order, solving this problem for $k = 2$ and giving a construction which we conjecture to be extremal for larger $k$. We close with some results on generalised Tur\'{a}n problems for the number of directed cycles and paths in $k$-geodetic digraphs.
		
	\end{abstract}
	
	
	\maketitle
	
	%----------------------------------------------
	
	\section{Introduction}
	
	Tur\'{a}n problems are a fundamental part of extremal combinatorics.  Such a problem typically asks for the largest possible size of a graph $G$ with a family $\mathcal{F}$ of forbidden subgraphs. When $\mathcal{F}$ consists of small cycles, this is equivalent to the girth problem. Erd\H{o}s conjectured in 1975 that the largest possible size of a graph with order $n$ and girth $\geq 5$ is given by $(\frac{1}{2}+o(1))^{3/2}n^{3/2}$~\cite{Erd}; this conjecture remains open. Lazebnik and Ustimenko gave a construction of dense graphs with arbitrarily large girth in~\cite{LazUst}; the latest computational results for the problem can be found in~\cite{Afzaly}. 
	
	It is natural to extend this problem to directed graphs by asking for the largest possible size of a strongly connected digraph with order $n$ and no directed cycles of length $\leq g$. This problem was solved by Bermond et al. in~\cite{BerGerHeySot}. For the degree-restricted case, the well-known Caccetta-H\"aggkvist Conjecture states that the girth of a digraph with order $n$ and minimum out-degree $r$ is at most $\lceil \frac{n}{r}\rceil $.
	
	\begin{theorem}[\cite{BerGerHeySot}]
		Let $D$ be a strong digraph of order $n$, size $m$ and girth $g$.  Let $k \geq 2$. Then if
		\[ m \geq \frac{1}{2}(n^2+(3-2k)n+k^2-k), \]
		we must have $g \leq k$. This bound is tight.	
	\end{theorem}
	In this paper we consider an analogous problem using a different `girth-like' parameter. A digraph $G$ is \emph{$k$-geodetic} if for any pair of not necessarily distinct vertices $u,v$ there is at most one walk in $G$ from $u$ to $v$ with length $\leq k$. The \emph{geodetic girth} of $G$ is the largest value of $k$ such that $G$ is $k$-geodetic, when this value is defined (a bipartite graph has $k$-geodetic orientations for arbitrarily large $k$). By way of motivation, observe that all orientations of a graph with girth $\geq 2k+1$ are $k$-geodetic. The geodetic girth is of interest in a directed analogue of the degree/girth problem~\cite{Tui}. A related problem was considered in the papers~\cite{HuaLyu,HuaLyuQiao,Lyu,Wu}, which, for fixed $k \geq 2$, determine the largest size of a digraph with given order such that for each pair of vertices $u,v$ there is at most one $u,v$-walk with length exactly $k$. 
	
	In the papers~\cite{ShaUst,UstKoz}, Ustimenko et al. prove that if $f(n,k)$ is the largest size of a diregular $k$-geodetic digraph with order $n$, then for fixed $k$ we have $f(n,k) \sim n^{\frac{k+1}{k}}$. They also give a family of digraphs, now known as the \emph{permutation digraphs}, which meet this asymptotic bound. These digraphs were introduced in~\cite{Fio} and some further properties of these digraphs are given in~\cite{BruFioFio}. For $d,k \geq 2$ the permutation digraph $P(d,k)$ is defined as follows.  The vertices of $P(d,k)$ are all permutations $x_0x_1\dots x_{k-1}$ of length $k$ of symbols from the set $[d+k] = \{ 0, 1, \dots, d+k-1\} $.  A vertex $x_0x_1\dots x_{k-1}$ has an arc to all permutations of the form $x_1x_2\dots x_{k-1}x_k$ for any $x_k \not \in \{ x_0,x_1,\dots, x_{k-1}\}$. It is simple to verify that $P(d,k)$ is diregular with out-degree $d$ and is $k$-geodetic.  The order of $P(d,k)$ is $n = (d+k)(d+k-1)\cdots (d+1)$ and has size $m = nd \sim n^{\frac{k+1}{k}}$. We will see in the final section that these digraphs have other interesting extremal properties. The result of Ustimenko et al. also holds in the more general setting of out-regular digraphs as we can see from the following short argument.
	
	\begin{remark}
		For $k \geq 2$ the largest size $\ex_{out}(n,k)$ of an out-regular $k$-geodetic digraph with order $n$ satisfies
		$\ex_{out}(n;k) \sim n^{\frac{k+1}{k}}$ as $n \rightarrow \infty $.
	\end{remark}
	\begin{proof}
		It is known that the order $n$ of a $k$-geodetic digraph with minimum out-degree $d$ is bounded below by the directed \emph{Moore bound} $M(d,k) = 1 + d + d^2 + \dots + d^k$ (see~\cite{MilSir}).  Hence $n \geq d^k$ and, rearranging, $d \leq n^{1/k}$.  The size $m$ of an out-regular $k$-geodetic digraph $G$ with order $n$ thus satisfies $m = nd \leq n^{\frac{k+1}{k}}$.
	\end{proof}             
	In this paper we consider this problem without the restriction of diregularity. This problem can be put into the form of a forbidden subgraph problem, as every violation of $k$-geodecity in $G$ can be identified with the occurrence of a specific subdigraph of $G$; in~\cite{UstKoz} these subdigraphs are referred to as `hooves' or `commutative diagrams'. In Section~\ref{section:largest size} we find the largest size of a $k$-geodetic digraph with given order $n$ and classify the extremal digraphs. We then discuss the more difficult problem of finding the largest size of a strongly connected $k$-geodetic digraph; we solve this problem for $k = 2$ and give constructions that we conjecture to be extremal for larger $k$ in Section~\ref{section:largest size}. We classify the extremal digraphs for $k = 2$ in Section~\ref{section:classification}. Finally in Section~\ref{section:generalised Turan} we study some generalised Tur\'{a}n problems for $k$-geodetic digraphs.     
	
	A few words concerning notation. If there is an arc in a digraph $G$ from $u$ to $v$, then we write $u \rightarrow v$. The distance $d(u,v)$ between vertices $u,v$ of a digraph $G$ is the length of a shortest directed path from $u$ to $v$ in $G$ (or $\infty $ if no such path exists); note that we may have $d(u,v) \not = d(v,u)$ in a digraph. For a vertex $u$ of $G$ the \emph{out-neighbourhood} $N^+(u)$ of $u$ is defined to be $\{ v \in V(G): u \rightarrow v\}$; similarly the \emph{in-neighbourhood} of $u$ is $N^-(u) = \{ v \in V(G) : v \rightarrow u\} $. For a vertex $u$ of $G$ we define $N^{+k}(u):=\left\{v \in V(G): d(u,v)=k \right\}$ and similarly $N^{-k}(u):=\left\{v \in V(G): d(v,u)=k \right\}$.  The \emph{out-degree} $d^+(u)$ of a vertex $u$ of $G$ is the number of out-neighbours of $u$, i.e. $d^+(u) = |N^+(u)|$, and the \emph{in-degree} of $u$ is $d^-(u) = |N^-(u)|$. A vertex with out-degree zero is a \emph{sink} and a vertex with in-degree zero is a \emph{source}. For any other graph-theoretical terminology not defined here we follow~\cite{BonMur}.

	\section{The largest size of a $k$-geodetic digraph}\label{section:largest size}
	In this section we first classify the $2$-geodetic digraphs with given order and largest size without the assumption of diregularity; the corresponding result for larger $k$ follows immediately. We then solve the more difficult problem of the largest size of strongly connected $2$-geodetic digraphs with given order and give upper and lower bounds for the extremal size for larger $k$.
	
	\begin{theorem}\label{basic theorem}
		For $n \geq 4, k \geq 2$, the largest size of a $k$-geodetic digraph with order $n$ is $\left \lfloor \frac{n^2}{4} \right\rfloor $. 
	\end{theorem}
	\begin{proof}
		Orienting all edges of the complete bipartite graph $K_{\left \lceil n/2 \right \rceil , \left \lfloor n/2 \right \rfloor }$ in the same direction yields a $k$-geodetic digraph; this gives the required lower bound. Equality is easily checked for $4 \leq n \leq 6$.
		
		For the upper bound, consider first the graph $K_4^-$ consisting of the complete graph $K_4$ with one edge deleted. We claim that no graph containing a copy of $K_4^-$ has a $k$-geodetic orientation for $k \geq 2$. Suppose for a contradiction that a graph $G$ contains two triangles $x,y,z$ and $x,y,z'$, where $z \not = z'$.  The only 2-geodetic orientation of a triangle is a directed 3-cycle, so we can assume that $x \rightarrow y \rightarrow z \rightarrow x$ and $x \rightarrow y  \rightarrow z' \rightarrow x$; however there are now two distinct directed paths from $y$ to $x$ of length two, violating 2-geodecity. A simple inductive argument shows the well-known result~\cite{Dir,Erd4} that for $n \geq 7$, any graph with order $n$ and size $> \left \lfloor \frac{n^2}{4} \right \rfloor$ contains a copy of $K_4^-$ and the unique $K_4^-$-free graph with size $\left \lfloor \frac{n^2}{4} \right \rfloor$ is $K_{\lceil \frac{n}{2} \rceil ,\lfloor \frac{n}{2} \rfloor }$. As a $k$-geodetic digraph with $k \geq 3$ is also $2$-geodetic, the same upper bound applies for larger $k$. We therefore have equality for $k \geq 2$ and $n \geq 4$. 
	\end{proof} 
	
	We now classify the $2$-geodetic digraphs that meet the bound in Theorem~\ref{basic theorem}.  
	
	\begin{lemma}\label{classificationlemma}
		Let $K$ be a 2-geodetic orientation of a complete bipartite graph $K_{s,t}$ with partite sets $X$ and $Y$, where $s \geq t \geq 2$. If $x$ is any vertex of $K$ that is neither a source nor a sink, then either $d^+(x) = 1$ or $d^-(x) = 1$.
	\end{lemma}
	\begin{proof}
		Let $x \in X$ be a vertex of $K$ that is neither a source nor a sink.  Suppose that $d^+(x) \geq 2$ and $d^-(x) \geq 2$. Let $y \in Y$ be an out-neighbour of $x$ such that $y$ is not a sink. Hence $y \rightarrow x'$ for some $x' \in X-\{x\}$. If any other out-neighbour $y'$ of $x$ has an arc to $x'$, then we would have two 2-paths $x \rightarrow y \rightarrow x'$ and $x \rightarrow y' \rightarrow x'$, violating 2-geodecity, so it follows that $x'$ has arcs to every vertex of $N^+(x)-\{y\} $. Any in-neighbour $y^-$ of $x$ can already reach every vertex of $N^+(x)$ by a 2-path via $x$. As $x'$ has arcs to every vertex of $N^+(x)-\{ y\} $, it follows that $x' \rightarrow y^-$ for every in-neighbour $y^-$ of $x$. However, there are at least two such in-neighbours $y_1^-$ and $y_2^-$ by assumption, so there exist paths $x' \rightarrow y_1^- \rightarrow x$ and $x' \rightarrow y_2^- \rightarrow x$, a contradiction.
		
		It follows that every out-neighbour of $x$ in $Y$ is a sink and similarly every in-neighbour of $x$ is a source.  Let $x^* \in X-\{ x\} $. Then if $y^+ \in N^+(x), y^- \in N^-(x)$ we have two paths $y^- \rightarrow x \rightarrow y^+$ and $y^- \rightarrow x^* \rightarrow y^+$, which is impossible.  Hence we must have either $d^+(x) = 1$ or $d^-(x) = 1$.
	\end{proof}
	
	\begin{theorem}\label{completebipartitegraphs}
		For $n \geq 7$ the 2-geodetic digraphs with largest size are isomorphic to an orientation of $K_{\lceil \frac{n}{2} \rceil, \lfloor \frac{n}{2} \rfloor }$ with all arcs oriented in the same direction, except for a matching that is oriented in the opposite direction. The number of isomorphism classes of extremal digraphs is $n+1$ for odd $n \geq 7$ and $\frac{n}{2}+1$ for even $n \geq 8$.
	\end{theorem}
	\begin{proof}
		Let $K$ be a 2-geodetic orientation of $K_{\lceil \frac{n}{2}\rceil ,\lfloor \frac{n}{2} \rfloor }$; call the partite sets $X$ and $Y$, where $|X| \geq |Y|$. If $X$ contains a source then $Y$ contains no sources and vice versa, so we can assume that any source of $G$ lies in $X$. If $X$ contains only sources then we recover the construction in Theorem~\ref{basic theorem}, so we can assume that $X$ contains a vertex that is neither a source nor a sink. As $K$ is $2$-geodetic, $X$ cannot contain both sources and sinks; for example if $x_1 \in X$ is a source and $x_2 \in X$ is a sink, then if $y_1,y_2 \in Y$ we have paths $x_1 \rightarrow y_1 \rightarrow x_2$ and $x_1 \rightarrow y_2 \rightarrow x_2$, which is impossible. 
		
		Note that if all vertices are neither a source nor a sink then both partitions contain a vertex which is neither a source nor a sink. If $s=t=2$ then by Lemma~\ref{classificationlemma} we are done, so we may assume that $|Y|>2$ and $X$ contains a vertex, say $x_1$, which is neither a source nor a sink. By Lemma~\ref{classificationlemma} any such vertex has either in-degree or out-degree one; without loss of generality, we assume that $N^+(x_1) = \{ y_1\}$. 
		
		Suppose that $X$ contains a vertex $x'$ with $d^+(x') > 1$. If $x'$ has two out-neighbours $y',y'' \in Y-\{ y_1\} $, then $x' \rightarrow y' \rightarrow x_1$ and $x' \rightarrow y'' \rightarrow x_1$ would be two distinct $x',x_1$-paths of length two. Hence we can assume that $N^+(x')= { y_1,y} $ for some $y \in Y-\{y_1\} $. Therefore for any $y' \in Y-\{ y_1,y\} $ we have $y' \rightarrow x_1$ and $y' \rightarrow x'$, so that $y'$ has two paths of length two to $y_1$, a contradiction. Applying Lemma~\ref{classificationlemma}, we have the desired result. 
	\end{proof}

	%Now we move on to strongly connected digraphs
	
	For even $n$, if the matching mentioned in Theorem~\ref{completebipartitegraphs} is chosen to be a perfect matching, then the resulting digraph is strongly connected; however, all of the extremal digraphs in Theorem~\ref{completebipartitegraphs} for odd $n$ contain either a source or a sink. It is therefore natural to ask for the largest size of a strongly-connected $k$-geodetic digraph with given order. 
	
	\begin{definition}
		For $n \geq k+1$ and $k \geq 2$, $\ex(n;k)$ is the largest possible size of a strongly-connected $k$-geodetic digraph with order $n$.
	\end{definition}
	
	We now determine the exact value of $\ex(2r+1;2)$; we classify the extremal digraphs for odd $n$ in Section~\ref{section:classification}.  Taking a strongly connected 2-geodetic digraph with order $2r$ and size $r^2$ and expanding one arc into a directed triangle shows that $\ex(2r+1;2) \geq r^2+2$ (this construction is shown in Figure~\ref{fig:largeconstruction}).  We now show that this lower bound is optimal and that any $2$-geodetic digraph with odd order and larger size contains either a source or a sink.
	
	\begin{figure}\centering
		\begin{tikzpicture}[midarrow=stealth,x=0.2mm,y=-0.2mm,inner sep=0.1mm,scale=1.5,
			thick,vertex/.style={circle,draw,minimum size=10,font=\tiny,fill=white},edge label/.style={fill=white}]
			\tiny
			\node at (-50,-50) [vertex] (z) {};	
			\node at (0,0) [vertex] (u0) {};
			\node at (0,-100) [vertex] (v0) {};
			\node at (50,0) [vertex] (u1) {};
			\node at (50,-100) [vertex] (v1) {};	
			\node at (100,0) [vertex] (u2) {};
			\node at (100,-100) [vertex] (v2) {};	
			\node at (150,0) [vertex] (u3) {};
			\node at (150,-100) [vertex] (v3) {};		
			
			\path
			(v0) edge [midarrow,ultra thick] (z)
			(z) edge [midarrow,ultra thick] (u0)
			(u0) edge [midarrow,ultra thick] (v0)
			
			(u1) edge [midarrow] (v1)	
			(u2) edge [midarrow] (v2)
			(u3) edge [midarrow] (v3)
			(v0) edge [midarrow] (u1)
			(v0) edge [midarrow] (u2)
			(v0) edge [midarrow] (u3)
			(v1) edge [midarrow] (u0)
			(v1) edge [midarrow] (u2)
			(v1) edge [midarrow] (u3)	
			(v2) edge [midarrow] (u0)
			(v2) edge [midarrow] (u1)
			(v2) edge [midarrow] (u3)
			(v3) edge [midarrow] (u0)
			(v3) edge [midarrow] (u2)
			(v3) edge [midarrow] (u1)			
			;

		\end{tikzpicture}
		\caption{A strongly connected digraph with $n = 2r+1$ and $m = r^2+2$ for $r = 4$ with the triangle in bold}
		\label{fig:largeconstruction}
	\end{figure}
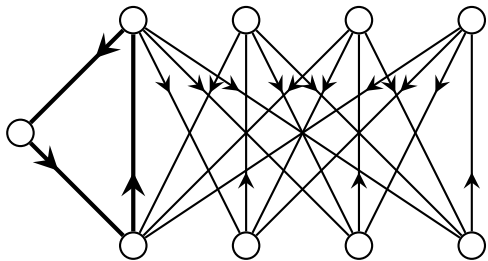

	\begin{theorem}\label{thm:2_Geo_Ext_Number}
		For all $r \geq 1$ we have $\ex(2r+1;k) = r^2+2$, and any 2-geodetic digraph with larger size contains a source or a sink.  For $r \geq 3$, the underlying graph of any 2-geodetic digraph $G$ with order $n = 2r+1$ and size $m = r^2+2$ that has no sources or sinks is isomorphic to a graph formed from a triangle $T$ and a copy of $K_{r-1,r-1}$ by joining every vertex in $K_{r-1,r-1}$ to exactly one vertex of $T$.  
	\end{theorem}    
	\begin{proof}
		The result for $r \leq 2$ follows easily by computer search, so we can assume that $r \geq 3$. Let $G$ be a $2$-geodetic digraph with size $m \geq r^2+2$ and $H$ be the underlying undirected graph of $G$; then $H$ is $K_4^-$-free. We will proceed to show that $H$ contains a triangle with a special substructure. 
		
		Suppose that $H$ is triangle-free. As the size of $H$ is at least $r^2+2$, it follows by the stability results of~\cite{Erd2,Sim} that $H$ is bipartite. We will name the larger partite set $X$ and the smaller $Y$. We claim that there are at least three vertices of $X$ that are connected to every vertex of $Y$.  Otherwise, setting $t = |X|$, the size of $H$ is bounded above by 
		\[ f(t) = (t-2)(2r-t)+2(2r+1-t) = -t^2+2rt+2,\] 
		where $r+1 \leq t \leq 2r+1$ as $|X| > |Y|$. The function $f(t)$ has its maximum at $t = r+1$, where $f(r+1) = r^2+1 < m$.
		
		Let $X'$ be the set of vertices in $X$ that are adjacent to every vertex of $Y$ and let $H'$ be the complete bipartite subgraph of $H$ with partite sets $X'$ and $Y$.  By Theorem~\ref{completebipartitegraphs} all edges of $H$ between $X$ and $Y$ are directed in the same direction with the exception of a matching $M$ of size $\leq r$ in the opposite direction.  Taking the converse of $G$ if necessary we can assume that the edges of $M$ are directed from $Y$ to $X$, with all other edges in the other direction. By assumption $G$ contains no sources or sinks, so every vertex of $X'$ must be incident with an edge of the matching $M$. 
		
		Therefore if two vertices $y,y'$ of $Y$ have a common out-neighbour $x$, then there will be a vertex $x' \in X'$ that has arcs to both $y$ and $y'$ and hence has two paths of length two to $x$, violating 2-geodecity. Hence the out-neighbourhoods of the vertices of $Y$ are pairwise disjoint. No vertex of $X$ is a source and so each vertex of $X$ has at least one in-neighbour in $Y$. As $|X| > |Y|$ there must be a vertex $y \in Y$ that has (at least) two out-neighbours $x_1$ and $x_2$ in $X$. If $x_1$ and $x_2$ had a common out-neighbour $y' \in Y$ then there would be two paths of length two from $y$ to $y'$, so we have $N^+(x_1) \cap N^+(x_2) = \emptyset $.  Hence there are at most two arcs incident to $\{ x_1,x_2\} $ and at most $r-1$ arcs incident from $\{ x_1,x_2\} $, so there are at most $r+1$ arcs incident with $x_1$ or $x_2$.  If we delete $x_1$ and $x_2$ we would thus obtain a 2-geodetic digraph with order $n = 2(r-1) \geq 4$ and size $m'$, where
		\[ m' \geq m-r-1 \geq r^2-2+1 > (r-1)^2,\]
		contradicting Theorem~\ref{basic theorem}. Therefore $H$ contains a triangle $T$.    
		
		Let us label the vertices of $T$ by $x,y,z$. As $G$ is $K_4^-$-free every vertex of $V(G)-T$ is adjacent to at most one vertex of $T$, so deleting $T$ from $G$ removes at most $n = 2r+1$ arcs. By Theorem~\ref{basic theorem} the size of $G-T$ is at most $(r-1)^2$. Thus $m \leq (2r+1)+(r-1)^2 = r^2+2$ and equality holds only if $G-T$ is an extremal digraph given in Theorem~\ref{completebipartitegraphs} and every vertex of $H-T$ is adjacent with exactly one vertex of $T$. 
	\end{proof}

	We turn now to the question of the largest size of strongly-connected $k$-geodetic digraphs for $k \geq 3$. It is trivial to provide a stronger upper bound on $\ex(n;k)$ than Theorem~\ref{basic theorem} for $k \geq 5$.
	
	\begin{lemma}
		For $k \geq 5$ we have $\ex(n;k) < \frac{n^2}{k}$.	
	\end{lemma}	
	\begin{proof}
		Let $G$ be a $k$-geodetic digraph without sinks. Suppose that $G$ contains a vertex $u$ with out-degree $d^+(u) \geq \frac{n}{k}$. As every vertex has out-degree at least one, it follows that $|N^{+t}(u)| \geq d^+(u) = \frac{n}{k}$ for $1 \leq t \leq k$, where $N^{+t}(u)$  denotes the set of vertices at distance $t$ from $u$. As $G$ is $k$-geodetic, all of the vertices in these sets are distinct, so it follows that $n \geq 1 +  k\frac{n}{k}$, a contradiction. Hence the maximum out-degree of $G$ is $\Delta ^+ < \frac{n}{k}$ and, summing over all vertices of $G$, the size of $G$ is $m < n\frac{n}{k} = \frac{n^2}{k}$.   	
	\end{proof}
	It follows that $\displaystyle\limsup _{n \rightarrow\infty }\frac{\ex(n;k)}{n^2} \leq \frac{1}{k}$. We now provide a construction that shows that \sloppy ${\displaystyle\frac{1}{k^2} \leq \liminf _{n \rightarrow\infty }\frac{\ex(n;k)}{n^2}}$.
	
	\begin{definition}
		Let the quotient and remainder when $n$ is divided by $k$ be $r$ and $s$ respectively, i.e. $n = kr+s$, where $s \leq r$ (this will hold for sufficiently large $n$).  The vertex set of $G(n,k)$ consists of vertices $u_{i,j}$ for $1 \leq i \leq r$ and $1 \leq j \leq k$, as well as $s$ further vertices $v_1,v_2,\dots v_s$.  We define the adjacencies of $G(n,k)$ as follows. 
		
		\begin{enumerate}[label=\roman*)]
			\item $u_{i,j} \rightarrow u_{i,j+1}$ for $1 \leq i \leq r$ and $1 \leq j \leq k-1$,
			\item $u_{i,k} \rightarrow v_i$ for $1 \leq i \leq s$,
			\item $u_{i,k} \rightarrow u_{j,2}$ for $s+1 \leq i \leq r$ and $1 \leq j \leq s$,
			\item $u_{i,k} \rightarrow u_{i',1}$ for $s+1 \leq i,i' \leq r$ and $i \not = i'$, and
			\item $v_t \rightarrow u_{i,1}$ for $1 \leq t \leq s$ and all $i$ in the range $1 \leq i \leq r$.
		\end{enumerate}
	\end{definition}
	
	The digraph $G(n,k)$ is $k$-geodetic and has size $m = rs+(k-1)r+s+(r-s)(r-1) = r^2+(k-2)r+2s$. If $r+1 \leq s \leq k-1$, then we have $\lfloor \frac{n}{k} \rfloor \leq k -2$, which is equivalent to $n \leq k^2-k-1$; thus these digraphs exist for $n \geq k^2-k$. The arcs in part~iii) can also be directed to $u_{j,1}$; combined with taking the converse of the resulting digraphs, this generates several different isomorphism classes. These digraphs admit a particularly simple description when $k|n$. Let $n = kr$ for some $r \geq 2$. Then $G(kr,k)$ is $k$-geodetic and has order $kr$ and size $r(r-1)+r(k-1) = r^2 + (k-2)r = \frac{n^2}{k^2}+\frac{(k-2)n}{k}$. It has vertices $u_{i,j}$, where $1 \leq i \leq r$ and $1 \leq j \leq k$ and contains the following arcs:
	\begin{enumerate}[label=\roman*)]
		\item $u_{i,j} \rightarrow u_{i,j+1}$ for $1 \leq i \leq r$ and $2 \leq j \leq k-1$,
		\item $u_{i,1} \rightarrow u_{i',2}$ for $1 \leq i,i' \leq r$ and $i \not = i'$, and
		\item $u_{i,k} \rightarrow u_{i,1}$ for $1 \leq i \leq r$.
	\end{enumerate}
	Observe that $G(kr,k)$ can be derived from the orientation of $K_{r,r}$ with a perfect matching pointing in one direction and all other arcs directed in the opposite direction by extending the perfect matching into paths of length $k-1$. The digraph $G(24;6)$ is shown in Figure~\ref{fig:Gkr}.

	\begin{figure}\centering
		\begin{tikzpicture}[midarrow=stealth,x=0.2mm,y=-0.2mm,inner sep=0.1mm,scale=1.85,
			thick,vertex/.style={circle,draw,minimum size=10,font=\tiny,fill=white},edge label/.style={fill=white}]
			\tiny
			\node at (-100,100) [vertex] (u00) {};
			\node at (-50,100) [vertex] (u01) {};		
			\node at (0,100) [vertex] (u02) {};			
			\node at (50,100) [vertex] (u03) {};
			\node at (100,100) [vertex] (u04) {};		
			\node at (150,100) [vertex] (u05) {};			
			
			\node at (-100,50) [vertex] (u10) {};
			\node at (-50,50) [vertex] (u11) {};		
			\node at (0,50) [vertex] (u12) {};			
			\node at (50,50) [vertex] (u13) {};
			\node at (100,50) [vertex] (u14) {};		
			\node at (150,50) [vertex] (u15) {};
			
			\node at (-100,0) [vertex] (u20) {};
			\node at (-50,0) [vertex] (u21) {};		
			\node at (0,0) [vertex] (u22) {};			
			\node at (50,0) [vertex] (u23) {};
			\node at (100,0) [vertex] (u24) {};		
			\node at (150,0) [vertex] (u25) {};
			
			\node at (-100,-50) [vertex] (u30) {};
			\node at (-50,-50) [vertex] (u31) {};		
			\node at (0,-50) [vertex] (u32) {};			
			\node at (50,-50) [vertex] (u33) {};
			\node at (100,-50) [vertex] (u34) {};		
			\node at (150,-50) [vertex] (u35) {};		
			\path
			(u04) edge [midarrow] (u03)
			(u03) edge [midarrow] (u02)
			(u02) edge [midarrow] (u01)
			(u01) edge [midarrow] (u00)
			
			(u14) edge [midarrow] (u13)
			(u13) edge [midarrow] (u12)
			(u12) edge [midarrow] (u11)
			(u11) edge [midarrow] (u10)
			
			(u24) edge [midarrow] (u23)
			(u23) edge [midarrow] (u22)
			(u22) edge [midarrow] (u21)
			(u21) edge [midarrow] (u20)
			
			(u34) edge [midarrow] (u33)
			(u33) edge [midarrow] (u32)
			(u32) edge [midarrow] (u31)
			(u31) edge [midarrow] (u30)		
			
			(u05) edge [midarrow] (u14)
			(u05) edge [midarrow] (u24)
			(u05) edge [midarrow] (u34)
			
			(u15) edge [midarrow] (u04)
			(u15) edge [midarrow] (u24)
			(u15) edge [midarrow] (u34)
			
			(u25) edge [midarrow] (u04)
			(u25) edge [midarrow] (u14)
			(u25) edge [midarrow] (u34)
			
			(u35) edge [midarrow] (u04)
			(u35) edge [midarrow] (u14)
			(u35) edge [midarrow] (u24)
			
			(u00) edge [midarrow,bend left = 20] (u05)
			(u10) edge [midarrow,bend left = 20] (u15)
			(u20) edge [midarrow,bend left = 20] (u25)
			(u30) edge [midarrow,bend left = 20] (u35)

			;

		\end{tikzpicture}
		\caption{$G(24,6)$}
		\label{fig:Gkr}
	\end{figure}		
	
	\begin{table}

		\begin{small}
			\begin{center}
				\begin{tabular}{| c||c|c| c| c|c|  }
					\hline
					
					$n/k$ & 3 & 4 & 5 & 6  \\
					\hline \hline
					7 & 8 &  &  &  \\
					\hline
					8 & 10 &  &  & \\
					\hline
					9 & 12 & 10 &  &  \\
					\hline
					10 & 14 & 12 &  &  \\
					\hline
					11 & 16 & 14 & 12 &  \\
					\hline
					12 & 20 & 15 & 14 &  \\
					\hline
					13 & 22 & 17 & 15 & 14 \\
					\hline
					14 & 24 & 19 & 17 & 16 \\
					\hline
					15 &  & 21 & 18 & 17 \\
					\hline
					16 &  &  & 20 & 19  \\
					\hline
					17 &  &  & 22 & 20  \\
					\hline
					18 &  &  &  & 21 \\
					\hline
					19 &  &  &  & 23  \\
					\hline
					
					\hline
				\end{tabular}
			\end{center}
		\end{small}
		\caption{$\ex(n;k)$ for some small values of $n$ and $k$}
		\label{fig:extable}
	\end{table}
	
	Table~\ref{fig:extable} displays the results of computational work on the values of $\ex(n;k)$ for some small values of $n$ and $k \geq 3$. It can be seen that the digraph $G(n,k)$ has largest possible size whenever $n = kr+s$, where $s \leq \min\{r,k-1\}$.  In fact for $n$ and $k$ in the above range such that $k|n$ we can say further that the underlying undirected graph of $G(n,k)$ is the unique graph with size $\frac{n^2}{k^2}+\frac{(k-2)n}{k}$ that has a strongly-connected $k$-geodetic orientation. This leads us to make the following conjecture that generalises Theorem~\ref{thm:2_Geo_Ext_Number}.
	
	\begin{conjecture}
		For $k \geq 2$ and sufficiently large $n$, 
		\[ \ex(n;k) = \left \lfloor \frac{n}{k} \right \rfloor ^2 - (k+2)\left \lfloor \frac{n}{k} \right \rfloor +2n. \]
		Also if $k|n$, then $G(n,k)$ is the unique extremal digraph with that order. 
	\end{conjecture}

	\section{Classification of extremal $2$-geodetic digraphs without sources and sinks}\label{section:classification}
	
	In the previous section it was shown that for $r \geq 1$, any strongly-connected $2$-geodetic digraph with order $n = 2r+1$ has at most $\ex(2r+1;2) = r^2+2$ arcs. In this section we will classify the strongly-connected $2$-geodetic digraphs that achieve this bound. Our analysis will focus on the case $r \geq 5$, i.e. odd $n \geq 11$. Computer search shows that there is a unique extremal strongly-connected $2$-geodetic digraph with size $r^2+2$ for $r = 1$, $3$ extremal digraphs for $r = 2$, $29$ solutions for $r = 3$ and $19$ solutions for $r = 4$; and any $2$-geodetic digraphs with larger size contain either a source or a sink.
	
	Let $G$ be a $2$-geodetic digraph with order $n = 2r+1 \geq 11$, size $r^2+2$ and no sources or sinks and let $H$ be the underlying undirected graph of $G$. By Theorem~\ref{thm:2_Geo_Ext_Number}, $H$ contains a triangle $T$ with vertices $x, y, z$, which is oriented in $G$ as $x \rightarrow y \rightarrow z \rightarrow x$ and each vertex in $H-T$ is adjacent to exactly one of $x, y$ or $z$.  Furthermore, $G-T$ must be one of the $r$ orientations of $K_{r-1,r-1}$ given in Theorem~\ref{completebipartitegraphs}.  Let the bipartition of $K_{r-1,r-1}$ be $X, Y$,  where $X = \{ x_1, \ldots , x_{r-1}\} , Y = \{ y_1, \ldots , y_{r-1}\} $; we can assume that $x_i \rightarrow y_i$ for $1 \leq i \leq r-1-s$ for some $0 \leq s \leq r-1$, with all other edges oriented in the other direction, so that there are $s$ sources and $s$ sinks in $G-T$.
	
	We will say that a partite set is \emph{covered} by a subset $T'$ of $T$ if all of its neighbours in $T$ belong to $T'$; in particular, if all of the neighbours of a partite set, say $X$, are the same vertex of $T$, say $x$, then $X$ is covered by $x$.
	We call a vertex in $T$ \emph{bad} if it has neighbours in both partite sets of $H-T$.
	
	\begin{lemma}\label{badvertices}
		Any bad vertex of $T$ has degree four in $H$.  If $n \geq 11$ then there is at most one bad vertex. 
	\end{lemma}
	\begin{proof}
		It is easily seen that if a bad vertex has degree $\geq 5$ in $H$, and hence $\geq 3$ neighbours in $K_{r-1,r-1}$, then $H$ contains a copy of $K_4^-$, which is impossible by Theorem~\ref{basic theorem}.  As any bad vertex of $T$ is adjacent to one vertex of $X$ and one vertex of $Y$, for $r \geq 5$ not all three vertices of $T$ can be bad. Furthermore if two vertices of $T$ are bad then the third vertex of $T$ would also have to be bad.        
	\end{proof}
	
	\begin{lemma}
		If there is no bad vertex in $T$, then either $X$ or $Y$ is covered by a single vertex of $T$ and $s \leq 1$. 
	\end{lemma}
	\begin{proof}
		If there is no bad vertex, then the neighbours of any vertex of $T$ in $K_{r-1,r-1}$ must be entirely contained in one partite set, so one partite set is covered by one vertex of $T$ and the other partite set is covered by the other two vertices of $T$.
		
		Concerning the value of $s$, suppose that $s \geq 2$ and that $X$ is covered by $x$ (the argument for $Y$ is similar).  There must be an arc from every sink in $X$ to $x$.  But any source in $Y$ has arcs to all the sinks in $X$ and hence will have multiple 2-paths to $x$, contradicting 2-geodecity.
	\end{proof}
	
	\begin{lemma}\label{inneighboursinX}
		Any vertex of $T$ with neighbours in $X$ has at most one in-neighbour in $X$.  Any vertex of $T$ that is joined to $\geq 2$ non-sink vertices of $X$ has no in-neighbour among the non-sink vertices of $X$. Substituting `source' for `sink' and `out-neighbour' for `in-neighbour', the analogous results hold for $Y$.  
	\end{lemma}
	\begin{proof}
		Suppose that a vertex of $T$, say $x$, has $\geq 2$ in-neighbours $x_i$ and $x_j$ in $X$. For any $\ell \in \{ 1,2,\dots , r-1\} - \{ i,j\} $ we have $y_\ell \rightarrow x_i \rightarrow x$ and $y_\ell \rightarrow x_j \rightarrow x$, a contradiction.
		
		Now let $x$ be adjacent to vertices $x_i$ and $x_j$ in $X$, where we now assume that $x_i$ and $x_j$ are not sinks in $G-T$.  If $x_i \rightarrow x$, then as $x$ has at most one in-neighbour in $X$ we must have $x \rightarrow x_j$.  Hence there are paths $x_i \rightarrow x \rightarrow x_j$ and $x_i \rightarrow y_i \rightarrow x_j$, a contradiction.  The results for $Y$ follow in a similar manner.	\end{proof}

	First we shall deal with the case that $T$ has no bad vertices.  Assume firstly that $X$ is covered by $x$.  Suppose that $s = 1$ (Figure~\ref{fig:A6}).   Then $x_{r-1}$ and $y_{r-1}$ are the sink and the source of $G-T$ respectively. Now $x$ must have an arc from the sink so that it does not remain a sink in $G$; hence by Lemma~\ref{inneighboursinX} we have $x \rightarrow x_i$ for $1 \leq i \leq r-2$. $Y$ is covered by $y$ and $z$ and either $y$ or $z$ has an arc to the source $y_{r-1}$. 
	
	If $z$ has an arc to $Y$ then there would be multiple 2-paths from $z$ to a non-sink vertex in $X$ and similarly if $z$ has an in-neighbour $y_i$ in $Y$, then there would be 2-paths $y_i \rightarrow z \rightarrow x$ and $y_i \rightarrow x_{r-1} \rightarrow x$.  Therefore $z$ has no neighbours in $Y$, so $y$ must have an arc to $y_{r-1}$ and by Lemma \ref{inneighboursinX} $y_i \rightarrow y$ for $1 \leq i \leq r-2$. This yields the 2-geodetic digraph $A_r$, an example of which is shown in Figure~\ref{fig:A6}.  This digraph is isomorphic to its converse. 
	
	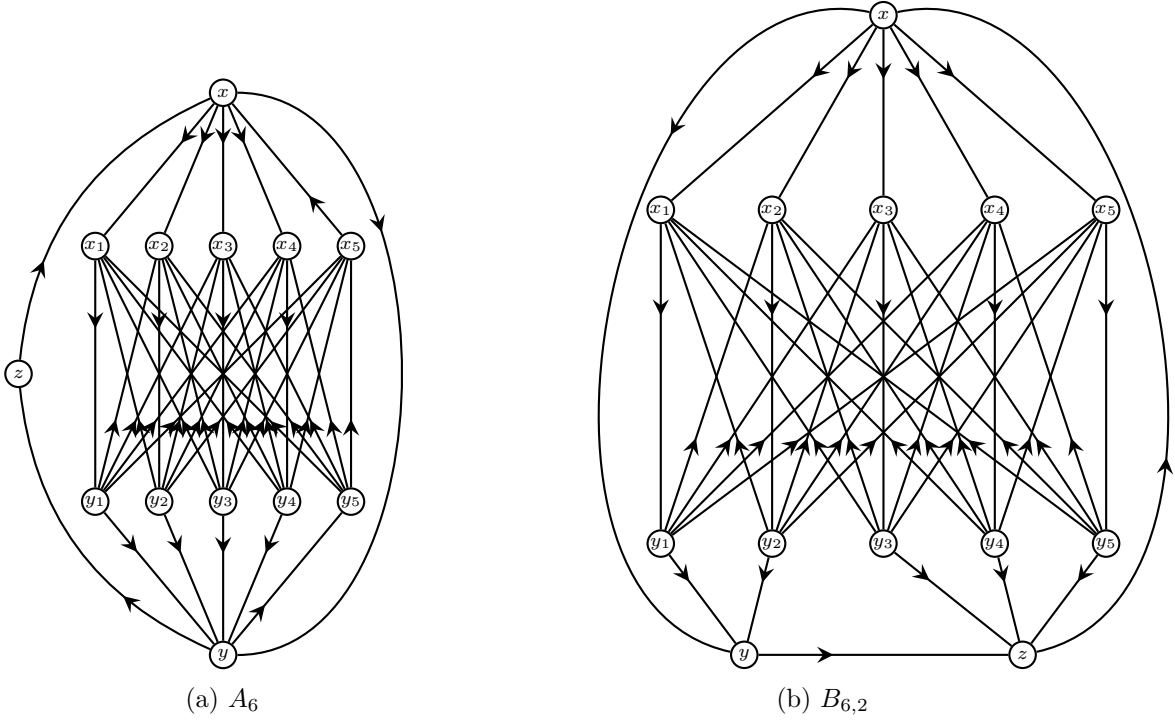
\begin{figure}\centering
		\begin{subfigure}[b]{0.35\textwidth}\centering
			\begin{tikzpicture}[midarrow=stealth,x=0.2mm,y=-0.2mm,inner sep=0.1mm,scale=1.70,
				thick,vertex/.style={circle,draw,minimum size=10,font=\tiny,fill=white},edge label/.style={fill=white}]
				\tiny
				
				\node at (50,-110) [vertex] (x) {$x$};
				\node at (-30,0) [vertex] (z) {$z$};
				\node at (50,110) [vertex] (y) {$y$};	
				
				\node at (0,-50) [vertex] (x1) {$x_1$};
				\node at (25,-50) [vertex] (x2) {$x_2$};
				\node at (50,-50) [vertex] (x3) {$x_3$};
				\node at (75,-50) [vertex] (x4) {$x_4$};
				\node at (100,-50) [vertex] (x5) {$x_5$};					
				\node at (0,50) [vertex] (y1) {$y_1$};
				\node at (25,50) [vertex] (y2) {$y_2$};
				\node at (50,50) [vertex] (y3) {$y_3$};
				\node at (75,50) [vertex] (y4) {$y_4$};
				\node at (100,50) [vertex] (y5) {$y_5$};		
				
				\path
				
				(x) edge [midarrow, bend left=90] (y)
				(y) edge [midarrow, bend left] (z)
				(z) edge [midarrow, bend left] (x)
				
				(y) edge [midarrow] (y5)
				(y1) edge [midarrow] (y)
				(y2) edge [midarrow] (y)
				(y3) edge [midarrow] (y)
				(y4) edge [midarrow] (y)			
				
				(x5) edge [midarrow] (x)
				(x) edge [midarrow] (x1)
				(x) edge [midarrow] (x2)
				(x) edge [midarrow] (x3)
				(x) edge [midarrow] (x4)
				
				(x1) edge [midarrow] (y1)	
				(x2) edge [midarrow] (y2)
				(x3) edge [midarrow] (y3)
				(x4) edge [midarrow] (y4)
				(y5) edge [midarrow] (x5)	
				
				(y1) edge [midarrow] (x2)
				(y1) edge [midarrow] (x3)	
				(y1) edge [midarrow] (x4)	
				(y1) edge [midarrow] (x5)	
				
				(y2) edge [midarrow] (x1)
				(y2) edge [midarrow] (x3)	
				(y2) edge [midarrow] (x4)	
				(y2) edge [midarrow] (x5)
				
				(y3) edge [midarrow] (x2)
				(y3) edge [midarrow] (x1)	
				(y3) edge [midarrow] (x4)	
				(y3) edge [midarrow] (x5)	
				
				(y4) edge [midarrow] (x2)
				(y4) edge [midarrow] (x3)	
				(y4) edge [midarrow] (x1)	
				(y4) edge [midarrow] (x5)
				
				(y5) edge [midarrow] (x2)
				(y5) edge [midarrow] (x3)	
				(y5) edge [midarrow] (x4)	
				(y5) edge [midarrow] (x1)

				;

			\end{tikzpicture}
			\caption{$A_6$}
			\label{fig:A6}
		\end{subfigure}
		\quad
		\begin{subfigure}[b]{0.55\textwidth}\centering
			\begin{tikzpicture}[midarrow=stealth,x=0.2mm,y=-0.2mm,inner sep=0.1mm,scale=1.85,
				thick,vertex/.style={circle,draw,minimum size=10,font=\tiny,fill=white},edge label/.style={fill=white}]
				\tiny
				
				\node at (80,-130) [vertex] (x) {$x$};
				\node at (30,100) [vertex] (y) {$y$};
				\node at (130,100) [vertex] (z) {$z$};	
				
				\node at (0,-60) [vertex] (x1) {$x_1$};
				\node at (40,-60) [vertex] (x2) {$x_2$};
				\node at (80,-60) [vertex] (x3) {$x_3$};
				\node at (120,-60) [vertex] (x4) {$x_4$};
				\node at (160,-60) [vertex] (x5) {$x_5$};					
				\node at (0,60) [vertex] (y1) {$y_1$};
				\node at (40,60) [vertex] (y2) {$y_2$};
				\node at (80,60) [vertex] (y3) {$y_3$};
				\node at (120,60) [vertex] (y4) {$y_4$};
				\node at (160,60) [vertex] (y5) {$y_5$};		
				
				\path
				
				(x) edge [midarrow, bend right=90] (y)
				(y) edge [midarrow] (z)
				(z) edge [midarrow, bend right=90] (x)
				
				(y1) edge [midarrow] (y)
				(y2) edge [midarrow] (y)
				(y3) edge [midarrow] (z)
				(y4) edge [midarrow] (z)
				(y5) edge [midarrow] (z)			
				
				(x) edge [midarrow] (x1)
				(x) edge [midarrow] (x2)
				(x) edge [midarrow] (x3)
				(x) edge [midarrow] (x4)
				(x) edge [midarrow] (x5)	
				
				(x1) edge [midarrow] (y1)	
				(x2) edge [midarrow] (y2)
				(x3) edge [midarrow] (y3)
				(x4) edge [midarrow] (y4)
				(x5) edge [midarrow] (y5)	
				
				(y1) edge [midarrow] (x2)
				(y1) edge [midarrow] (x3)	
				(y1) edge [midarrow] (x4)	
				(y1) edge [midarrow] (x5)	
				
				(y2) edge [midarrow] (x1)
				(y2) edge [midarrow] (x3)	
				(y2) edge [midarrow] (x4)	
				(y2) edge [midarrow] (x5)
				
				(y3) edge [midarrow] (x2)
				(y3) edge [midarrow] (x1)	
				(y3) edge [midarrow] (x4)	
				(y3) edge [midarrow] (x5)	
				
				(y4) edge [midarrow] (x2)
				(y4) edge [midarrow] (x3)	
				(y4) edge [midarrow] (x1)	
				(y4) edge [midarrow] (x5)
				
				(y5) edge [midarrow] (x2)
				(y5) edge [midarrow] (x3)	
				(y5) edge [midarrow] (x4)	
				(y5) edge [midarrow] (x1)

				;

			\end{tikzpicture}
			\caption{$B_{6,2}$}
			\label{fig:B62}
		\end{subfigure}
		\caption{The graphs $A_6$ and $B_{6,2}$}
		\label{fig:A6B62}
	\end{figure}

	Now let $X$ be covered by $x$ and $s = 0$.  By Lemma~\ref{inneighboursinX}, $x \rightarrow x_i$ for $1 \leq i \leq r-1$.  By reasoning similar to the previous case, $y$ and $z$ can have no out-neighbours in $Y$.  Let the resulting digraph in which $y$ has $t$ in-neighbours in $Y$ be denoted by $B_{r,t}$ for $0 \leq t \leq r-1$ (see Figure~\ref{fig:B62}). Each $B_{r,t}$ is a $2$-geodetic extremal digraph.

	The case of $Y$ being covered by one vertex of $T$ is symmetric.  In particular we shall denote the converse of $B_{r,t}$ by $B'_{r,t}$.  We have $B'_{r,0} \cong B_{r,0}$ and $B'_{r,r-1} \cong B_{r,r-1}$, but otherwise these digraphs are pairwise non-isomorphic. 
	
	We now turn to the case that there is a bad vertex; say $z$ is bad.  Hence $d(z) = 4$ in $H$.  It follows by Lemma~\ref{badvertices} that $x$ and $y$ each have $r-2$ neighbours in $K_{r-1,r-1}$ and each is connected to just one partite set.
	
	\begin{lemma}\label{badvertexstructure}
		If $z$ is bad, then $s \leq 2$.  If $z$ is joined to a source in $Y$, then $X$ is covered by $\{ y,z\}$ and $Y$ is covered by $\{ x,z\}$.  Likewise, if $z$ is joined to a sink in $X$, then $X$ is covered by $\{ x,z\}$  and $Y$ is covered by $\{ y,z\}$. If $s = 2$, then $z$ is connected to a source in $Y$ and a sink in $X$.
	\end{lemma}
	\begin{proof}
		Suppose that $s \geq 3$. The bad vertex $z$ is adjacent to one vertex of $Y$ in $H-T$, so the vertex of $T$ that also has edges to $Y$ must have arcs to two or more sources in $Y$, violating Lemma~\ref{inneighboursinX}. This reasoning also demonstrates that if $s = 2$, then $z$ is connected to a source in $Y$ and a sink in $X$. 	
		
		For any $s \leq 2$, suppose that $z$ is joined to a source in $Y$.  Suppose that $X$ is covered by $\{ x,z\}$. Then $z$ has a 2-path to every vertex of $X$ via the source, but by Lemma~\ref{inneighboursinX} $x$ has an out-neighbour $x_i \in X$, so there will also be a 2-path from $z$ to $x_i$ via $x$, violating 2-geodecity. Hence $X$ must be covered by $\{ y,z\}$ and hence $Y$ is covered by $\{ x,z\}$. The other statement is symmetric to this one.\end{proof}
	
	Let $s = 2$. The sources in $G-T$ are $y_{r-2}$ and $y_{r-1}$ and the sinks are $x_{r-2}$ and $x_{r-1}$.  By Lemma~\ref{badvertexstructure} we can assume that $z \rightarrow y_{r-2}$ and $x_{r-2} \rightarrow z$.  Also by Lemma~\ref{badvertexstructure} $X$ is covered by $\{ y,z\} $.  There must be an arc from $x_{r-1}$ to $y$ so that $x_{r-1}$ is not a sink in $G$ and $y \rightarrow x_i$ for $1 \leq i \leq r-3$ by Lemma~\ref{inneighboursinX}.  Likewise there is an arc from $x$ to $y_{r-1}$.  However, we now have two 2-paths from $x$ to the vertices in $\{ x_1, \dots, x_{r-3}\} $, one via $y$ and the other via $y_{r-1}$, a contradiction.  It follows that $s \leq 1$.
	
	Let $s=1$. The sink and source of $G-T$ are $x_{r-1}$ and $y_{r-1}$ respectively (Figure~\ref{fig:C6}).  Suppose that $z$ is joined to $x_{r-1}$ and $y_{r-1}$.  By Lemma~\ref{badvertexstructure} $X$ is covered by $\{ y,z\} $ and $Y$ is covered by $\{ x,z\} $.  By Lemma~\ref{inneighboursinX} $y \rightarrow x_i$ for $1 \leq i \leq r-2$ and  $y_i \rightarrow x$ for $1 \leq i \leq r-2$.
	%If $x$ has an out-neighbour $y_j$, $1 \leq j \leq r-2$, in $Y$, then $x$ would have multiple 2-paths to every vertex in $X- \{ x_j,x_{r-1}\} $ via $y_j$ and $y$.  Therefore $y_i \rightarrow x$ for $1 \leq i \leq r-2$.  
	This gives the single solution $C_r$, an example of which is shown in Figure~\ref{fig:C6}. Note that the digraph $C_r$ is isomorphic to its converse.

	\begin{figure}\centering
		\begin{subfigure}[b]{0.4\textwidth}\centering
			\begin{tikzpicture}[midarrow=stealth,x=0.2mm,y=-0.2mm,inner sep=0.1mm,scale=1.6,
				thick,vertex/.style={circle,draw,minimum size=10,font=\tiny,fill=white},edge label/.style={fill=white}]
				\tiny
				
				\node at (80,-90) [vertex] (y) {$y$};
				\node at (80,90) [vertex] (x) {$x$};
				\node at (-60,0) [vertex] (z) {$z$};	
				
				\node at (0,-45) [vertex] (x1) {$x_5$};
				\node at (30,-45) [vertex] (x2) {$x_1$};
				\node at (60,-45) [vertex] (x3) {$x_2$};
				\node at (90,-45) [vertex] (x4) {$x_3$};
				\node at (120,-45) [vertex] (x5) {$x_4$};					
				\node at (0,45) [vertex] (y1) {$y_5$};
				\node at (30,45) [vertex] (y2) {$y_1$};
				\node at (60,45) [vertex] (y3) {$y_2$};
				\node at (90,45) [vertex] (y4) {$y_3$};
				\node at (120,45) [vertex] (y5) {$y_4$};		
				
				\path
				
				(z) edge [midarrow] (y1)
				(x1) edge [midarrow] (z)
				
				(x) edge [midarrow, bend right=90] (y)
				(y) edge [midarrow, bend right] (z)
				(z) edge [midarrow, bend right] (x)

				(y2) edge [midarrow] (x)
				(y3) edge [midarrow] (x)
				(y4) edge [midarrow] (x)
				(y5) edge [midarrow] (x)

				(y) edge [midarrow] (x2)
				(y) edge [midarrow] (x3)
				(y) edge [midarrow] (x4)
				(y) edge [midarrow] (x5)	
				
				(y1) edge [midarrow] (x1)
				
				(x2) edge [midarrow] (y2)
				(x3) edge [midarrow] (y3)
				(x4) edge [midarrow] (y4)
				(x5) edge [midarrow] (y5)	
				
				(y1) edge [midarrow] (x2)
				(y1) edge [midarrow] (x3)	
				(y1) edge [midarrow] (x4)	
				(y1) edge [midarrow] (x5)	
				
				(y2) edge [midarrow] (x1)
				(y2) edge [midarrow] (x3)	
				(y2) edge [midarrow] (x4)	
				(y2) edge [midarrow] (x5)
				
				(y3) edge [midarrow] (x2)
				(y3) edge [midarrow] (x1)	
				(y3) edge [midarrow] (x4)	
				(y3) edge [midarrow] (x5)	
				
				(y4) edge [midarrow] (x2)
				(y4) edge [midarrow] (x3)	
				(y4) edge [midarrow] (x1)	
				(y4) edge [midarrow] (x5)
				
				(y5) edge [midarrow] (x2)
				(y5) edge [midarrow] (x3)	
				(y5) edge [midarrow] (x4)	
				(y5) edge [midarrow] (x1)

				;

			\end{tikzpicture}
			\caption{$C_6$}
			\label{fig:C6}
		\end{subfigure}
		\quad
		\begin{subfigure}[b]{0.4\textwidth}\centering
			\begin{tikzpicture}[midarrow=stealth,x=0.2mm,y=-0.2mm,inner sep=0.1mm,scale=1.6,
				thick,vertex/.style={circle,draw,minimum size=10,font=\tiny,fill=white},edge label/.style={fill=white}]
				\tiny
				
				\node at (80,-90) [vertex] (y) {$y$};
				\node at (80,90) [vertex] (x) {$x$};
				\node at (-60,0) [vertex] (z) {$z$};	
				
				\node at (0,-45) [vertex] (x1) {$x_1$};
				\node at (30,-45) [vertex] (x2) {$x_2$};
				\node at (60,-45) [vertex] (x3) {$x_3$};
				\node at (90,-45) [vertex] (x4) {$x_4$};
				\node at (120,-45) [vertex] (x5) {$x_5$};					
				\node at (0,45) [vertex] (y1) {$y_1$};
				\node at (30,45) [vertex] (y2) {$y_2$};
				\node at (60,45) [vertex] (y3) {$y_3$};
				\node at (90,45) [vertex] (y4) {$y_4$};
				\node at (120,45) [vertex] (y5) {$y_5$};		
				
				\path
				
				(z) edge [midarrow] (x1)
				(y1) edge [midarrow] (z)
				
				(x) edge [midarrow, bend right=90] (y)
				(y) edge [midarrow, bend right] (z)
				(z) edge [midarrow, bend right] (x)

				(y2) edge [midarrow] (x)
				(y3) edge [midarrow] (x)
				(y4) edge [midarrow] (x)
				(y5) edge [midarrow] (x)

				(y) edge [midarrow] (x2)
				(y) edge [midarrow] (x3)
				(y) edge [midarrow] (x4)
				(y) edge [midarrow] (x5)	
				
				(x1) edge [midarrow] (y1)
				
				(x2) edge [midarrow] (y2)
				(x3) edge [midarrow] (y3)
				(x4) edge [midarrow] (y4)
				(x5) edge [midarrow] (y5)	
				
				(y1) edge [midarrow] (x2)
				(y1) edge [midarrow] (x3)	
				(y1) edge [midarrow] (x4)	
				(y1) edge [midarrow] (x5)	
				
				(y2) edge [midarrow] (x1)
				(y2) edge [midarrow] (x3)	
				(y2) edge [midarrow] (x4)	
				(y2) edge [midarrow] (x5)
				
				(y3) edge [midarrow] (x2)
				(y3) edge [midarrow] (x1)	
				(y3) edge [midarrow] (x4)	
				(y3) edge [midarrow] (x5)	
				
				(y4) edge [midarrow] (x2)
				(y4) edge [midarrow] (x3)	
				(y4) edge [midarrow] (x1)	
				(y4) edge [midarrow] (x5)
				
				(y5) edge [midarrow] (x2)
				(y5) edge [midarrow] (x3)	
				(y5) edge [midarrow] (x4)	
				(y5) edge [midarrow] (x1)

				;

			\end{tikzpicture}
			\caption{$D_6$}
			\label{fig:D6}
		\end{subfigure}
		\caption{The digraphs $C_6$ and $D_6$}
		\label{fig:C6D6}
	\end{figure}
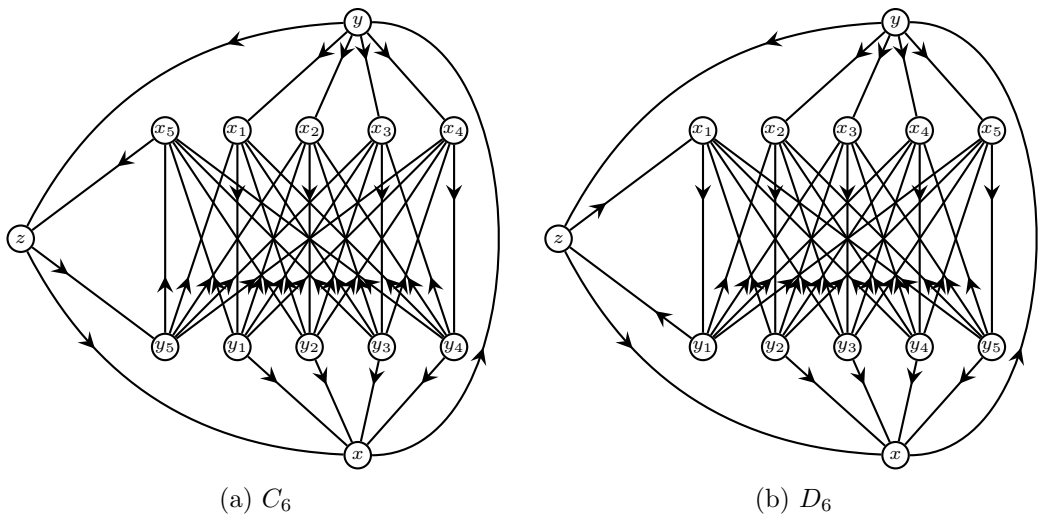
	
	Suppose that $z$ is joined to the source $y_{r-1}$ but is not joined to the sink $x_{r-1}$ of $G-T$; say $z$ has an edge to $x_{r-2}$ in $H-T$.  By Lemma~\ref{badvertexstructure} $X$ is covered by $\{ y,z\} $ and $Y$ is covered by $\{ x,z\} $.  Hence there is an arc $x_{r-1} \rightarrow y$ and by Lemma~\ref{inneighboursinX} $x$ has at most one out-neighbour in $Y-y_{r-1}$, so that there is a vertex $y_i$ with $y_i \rightarrow x$. Hence there would be paths $y_i \rightarrow x \rightarrow y$ and $y_i \rightarrow x_{r-1} \rightarrow y$ in $G$, a contradiction.  We will get a similar contradiction if $z$ is joined to the sink $x_{r-1}$ in $X$, but not to the source $y_{r-1}$ in $Y$.

	Finally let $z$ be joined to $x_i$ and $y_j$ where $1 \leq i,j \leq r-2$. Suppose that $X$ is covered by $\{ x,z\} $ and $Y$ by $\{ y,z\} $.  If $i=j$, then the triangle is oriented as $x_i \rightarrow y_i \rightarrow z \rightarrow x_i$; however this yields paths $y_i \rightarrow z \rightarrow x$ and $y_i \rightarrow x_{r-1} x$, so we must have $i \not = j$. Without loss of generality we can set $i = r-2$ and $j = r-3$. The triangle is now oriented as $y_{r-3} \rightarrow x_{r-2} \rightarrow z \rightarrow y_{r-3}$.  There is an arc $x_{r-1} \rightarrow x$ so by Lemma~\ref{inneighboursinX} there are arcs $x \rightarrow x_l$ for $1 \leq l \leq r-3$. In this case we would have paths $z \rightarrow y_{r-3} \rightarrow x_1$ and $z \rightarrow x \rightarrow x_1$.
	
	Hence we can assume that $X$ is covered by $\{ y,z\} $ and $Y$ by $\{ x,z\} $.  By Lemma~\ref{inneighboursinX} $y$ has at least two out-neighbours in $X$, so if $x$ has any out-neighbour in $Y$ then there would be more than one 2-path from $x$ to an out-neighbour of $y$ in $Y$.  In particular we must have $z \rightarrow y_{r-1}$, a case that we have already considered.

	Now we can set $s = 0$.  Suppose that $z$ is joined to $x_1$ and $y_2$.  As $y_2 \rightarrow x_1$, we must orient the triangle $z,x_1,y_2$ as $z \rightarrow y_2 \rightarrow x_1 \rightarrow z$.  If $X$ is covered by $\{ x,z\} $ and $Y$ is covered by $\{ y,z\} $, then by Lemma~\ref{inneighboursinX} $x \rightarrow x_i$ for $2 \leq i \leq r-1$ and so there would be paths $z \rightarrow y_2 \rightarrow x_3$ and $z \rightarrow x \rightarrow x_3$.  Hence $X$ must be covered by $\{ y,z\} $ and $Y$ must be covered by $\{ x,z\} $.  By Lemma~\ref{inneighboursinX} we have $y_1 \rightarrow x$.  Hence there are paths $x_1 \rightarrow y_1 \rightarrow x$ and $x_1 \rightarrow z \rightarrow x$.
	
	Therefore we can assume that $z$ is joined to $x_1$ and $y_1$.  We must have $z \rightarrow x_1$ and $y_1 \rightarrow z$.  If $X$ is covered by $\{ y,z\} $ and $Y$ by $\{ x,z\} $ then by Lemma~\ref{inneighboursinX} $y \rightarrow x_i$ and $y_i \rightarrow x$ for $2 \leq i \leq r-1$.  This yields the solution $D_r$ shown in Figure~\ref{fig:D6}.  $D_r$ is isomorphic to its converse. If $X$ is covered by $\{ x,z\} $ and $Y$ by $\{ y,z\} $ then by a suitable redrawing of the digraph it can be seen that we obtain a solution isomorphic to $C_r$ in Figure~\ref{fig:C6}.
	
	This completes our classification of the strongly-connected $2$-geodetic digraphs with order $n = 2r+1$ and size $r^2+2$. We therefore have the following theorem.

	\begin{theorem}\label{classificationtheorem}
		If $G$ is a 2-geodetic digraph with order $n = 2r+1 \geq 11$, size $m = r^2+2$ and no sources or sinks, then $G$ is either isomorphic to one of $A_r, B_{r,0},B_{r,r-1},C_r$ or $D_r$ or is isomorphic to a member of the family $B_{r,t},B'_{r,t}$ for some $1 \leq t \leq r-2$.  The digraphs in this list are pairwise non-isomorphic and so there are $2r+1$ distinct solutions up to isomorphism.  
	\end{theorem}

	\section{Generalised Tur\'{a}n problems for $k$-geodetic digraphs}\label{section:generalised Turan}
	
	Recently the following extension of Tur\'{a}n's problem has received a great deal of attention: given graphs $T$ and $H$ what is the largest possible number of copies of $T$ in an $H$-free graph with order $n$?  Erd\H{o}s considered this problem in 1962~\cite{Erd3} when $T$ and $H$ are complete graphs. The largest number of 5-cycles in a triangle-free graph was treated in~\cite{Grz,HatHlaKraNorRaz} and the converse problem of the largest number of triangles in a graph without a given odd cycle $C_{2k+1}$ is discussed in~\cite{BolGyo,GyoLi}. The problem was considered in greater generality in~\cite{AlonShik}. To investigate this problem in digraphs we define the following notation.
	
	\begin{definition}
		For any digraph $Z$ and $k \geq 2$ we denote the largest number of copies of $Z$ in a $k$-geodetic digraph by $\ex(n;Z;k)$.	
	\end{definition}
	Observe that if $Z$ is a directed arc then  $\ex(n;Z;k)=\ex(n;k)$. We will study the asymptotics of the function $\ex(n;Z;k)$ in the cases that $Z$ is a directed $(k+1)$-cycle or a directed path. We begin with the function $\ex(n;C_{k+1};k)$, where $k \geq 2$ and $C_{k+1}$ is a directed $(k+1)$-cycle. Earlier we made use of the fact that any arc in a $2$-geodetic digraph is contained in at most one triangle; a similar principle applies for larger $k$.

	\begin{lemma}
		Every arc in a $k$-geodetic digraph is contained in at most one directed $(k+1)$-cycle.
	\end{lemma}
	\begin{proof}
		Suppose that an arc $xy$ is contained in two distinct $(k+1)$-cycles.  Then $y$ has distinct paths of length $k$ to $x$, violating $k$-geodecity.
	\end{proof}
	
	We now utilise an inductive approach to give an upper bound on the number of directed $(k+1)$-cycles in a $k$-geodetic digraph.
	
	\begin{lemma}\label{rootnoutdegree}
		Every $k$-geodetic digraph with order $n$ contains a vertex with out-degree $\leq n^{1/k}$.
	\end{lemma}
	\begin{proof}
		Assume the contrary.  Then for any vertex $x$ the set $N^{+k}(x)$ will contain at least $n$ vertices, a contradiction. 
	\end{proof}

	\begin{theorem}
		\[ \ex(n;C_{k+1};k) \leq \sum_{i=1}^{n}i^{1/k} =  \frac{k}{k+1}n^{\frac{k+1}{k}} + O(n^{\frac{1}{k}}).\]
	\end{theorem}
	\begin{proof}
		We claim that 
		\[ \ex(n;C_{k+1};k) \leq \sum_{i=1}^{n}i^{1/k}. \]
		This is easily verified for small $n$, giving a basis for induction. Assume that the result is true for digraphs with order $n-1$ and consider a $k$-geodetic digraph $G$ with order $n$ and $\ex(n;C_{k+1};k)$ directed $(k+1)$-cycles and, subject to this, the smallest possible size $m$.  In particular, every arc of $G$ is contained in a unique $C_{k+1}$, for otherwise deleting this arc would yield a $k$-geodetic digraph with the same number of $(k+1)$-cycles but smaller size. It follows that we can pair off the in- and out-neighbours of every vertex according to the corresponding $(k+1)$-cycles.  Hence $d^-(x) = d^+(x)$ for every vertex $x$ of $G$ and every vertex $x$ is contained in exactly $d^-(x) = d^+(x)$ directed $(k+1)$-cycles.  
		
		By Lemma~\ref{rootnoutdegree}, $G$ contains a vertex $x$ with out-degree $\leq n^{1/k}$.  Deleting this vertex, we obtain a $k$-geodetic digraph with order $n-1$ which, by induction, contains at most $\Sigma _{i=1}^{n-1}i^{1/k} $ copies of $C_{k+1}$.  Deleting $x$ destroyed at most $n^{1/k}$ $(k+1)$-cycles, so the result follows by induction.
	\end{proof}
	
	In fact the upper bound is tight up to a multiplicative constant. We can show this using the permutation digraphs $P(d,k)$ that were discussed in Section~\ref{section:largest size}. The permutation digraph $P(d,k)$ has order $n = (d+k)(d+k-1)\ldots (d+1)$ and size $dn$.  It is easily seen that each arc of $P(d,k)$ is contained in a unique $(k+1)$-cycle; for example $0123\dots (k-1) \rightarrow 123\dots (k-1)k$ is contained in the unique $(k+1)$-cycle
	\[ 0123\dots (k-1) \rightarrow 123\dots (k-1)k \rightarrow 23\dots (k-1)k0 \rightarrow  \dots \rightarrow k0123\dots (k-2) \rightarrow 0123\dots (k-1).\]
	Hence $P(d,k)$ contains $\frac{nd}{k+1}$ copies of $C_{k+1}$.  Therefore asymptotically $\ex(n;C_{k+1};k)$ is at least $\frac{1}{k+1}n^{\frac{k+1}{k}}$. In particular, $ex(n;C_3;2)$ must lie somewhere between $\frac{1}{3}n^{3/2}$ and $\frac{2}{3}n^{3/2}$. We show that the lower bound is correct.
	
	\begin{theorem} \label{thm:triangle}
		\[ \ex(n;C_3;2) = \frac{1}{3}n^{3/2} + O(n^{\frac{1}{2}}).\]
	\end{theorem}
	\begin{proof}
		Let $G$ be a 2-geodetic digraph with order $n$ and $N: = \ex(n;C_3;2)$ directed triangles.  As before we can assume that every arc is contained in a unique triangle.  Thus $N = \frac{1}{3}\Sigma _{v \in G}d^+(v)$. For any vertex $v$ we have $\Sigma _{u \in N^+(v)}d^+(u) = |N^{+2}(v)| \leq n - 1- d^+(v)$.  By H\"{o}lder's inequality 
		\[ N = \frac{1}{3}\sum_{v \in G}d^+(v) \leq \frac{\sqrt{n}}{3}\sqrt{\sum_{v \in G} (d^+(v))^2}. \]
		
		In the sum $ \sum_{v \in G} \sum_{u \in N^+(v)} d^+(u)$ the term $d^+(u)$ appears $d^-(u) = d^+(u)$ times, so 
		
		\[ N \leq \frac{\sqrt{n}}{3}\sqrt{\sum_{v \in G} \sum_{u \in N^+(v)} d^+(u)} \leq  \frac{\sqrt{n}}{3}\sqrt{\sum_{v \in G} (n-1-d^+(v))} =  \frac{\sqrt{n}}{3}\sqrt{n^2-n-3N}. \]  
		Squaring both sides yields $N^2 \leq \frac{n}{9}(n^2-n-3N)$. 
		Rearranging and solving the associated quadratic equation, it follows that 
		\[ N \leq \left\lfloor \frac{n}{6}(\sqrt{4n-3}-1) \right\rfloor . \]
	\end{proof}
	
	\begin{remark}
		For infinitely many $n$, the upper bound in Theorem~\ref{thm:triangle} is at most $\frac{n}{3}$ off from the lower bound of the permutation digraph $P(d,2)$. 
	\end{remark}
	
	Based on this example, we make the following conjecture.
	
	\begin{conjecture}
		For all $k \geq 2$ we have
		\[ \ex(n;C_{k+1};k) = \frac{1}{k+1}n^{\frac{k+1}{k}} + O(n^{\frac{1}{k} }). \]
	\end{conjecture}
	
	We turn now to the problem of the largest number of directed paths of given length in a $2$-geodetic digraph. Let $P_{\ell}$ be the path of length $\ell $ (i.e. order $\ell+1$). Surprisingly there are some differences between odd and even length paths; in the following theorem we show different lower bounds.
	
	\begin{theorem}\label{Thm:Lowerbound_paths}
		If $k \geq 2$ and $k$ divides $\ell$, then we have
		\[
		\ex(n;P_{\ell };k) = n^{(\ell /k)+1}+ O(n^{1+\frac{\ell-1}{k}}).
		\]
		In particular, for every even $l$
		\[ \ex(n;P_{\ell };2) =  n^{(\ell/2)+1}+O(n^{\ell/2}).\]
		If $\ell $ is odd, we have 
		\[
		\ex(n;P_{\ell };2) \geq (n/2)^{(\ell+3)/2}.
		\]
		
	\end{theorem}
	\begin{proof}
		Let $\ell $ be even and let $P(d,k)$ be a permutation digraph with degree $d$. $P(d,k)$ has order $(d+k)(d+k-1)\dots (d+1)$.  From each vertex $x$ there are at least $d^k(d-1)(d-2)\dots(d-\ell+k) = d^l+O(d^{\ell-1})$ distinct $\ell $-paths with initial vertex $x$, so there are $d^{\ell+k} + O(d^{\ell+k-1})$ distinct $\ell $-paths in $P(d,k)$.  Thus there are $n^{(\ell /k)+1} + O(n^{1+\frac{\ell-1}{k}})$ distinct $\ell$-paths in $P(d,k)$. For an upper bound, consider a path of length $\ell $ with vertices $0, 1, \dots , \ell $.  By $k$-geodecity, given the two endpoints of a path of length $k$, all of the intermediate vertices are determined.  Hence we can only choose vertices $0, k, 2k, \ldots, \ell $ independently.  Hence $\ex(n;P_{\ell };k)$ is at most $n^{(\ell /k)+1}$.
		
		Now let $\ell $ be odd and consider an orientation of the complete bipartite graph $K_{r,r}$ where $n = 2r$, in which a perfect matching is oriented in one direction and all other arcs are oriented in the opposite direction.  We have already seen that this digraph is 2-geodetic. The $\frac{n}{2}$ vertices of one partite set are the initial vertices of $(\frac{n}{2})^{(\ell +1)/2} + O(n^{(\ell-1)/2})$ distinct $\ell $-paths, whereas the vertices in the other partite set are the initial vertices of only $O(n^{(\ell-1)/2})$ $\ell$-paths each.  Multiplying by $\frac{n}{2}$ yields the result. 
	\end{proof}
	As for paths of odd length, we have an asymptotically sharp result only for $P_3$. 
	
	\begin{theorem}
		$\ex(n;P_3;2) = (n/2)^{3}+O(n^2)$
	\end{theorem}
	\begin{proof}
		We have a lower bound from Theorem~\ref{Thm:Lowerbound_paths}. For an upper bound we denote the maximum out-degree by $\displaystyle\Delta:=\max_{v\in V(G)}\{d^+(v)\}$ and we assume without loss of generality that $\displaystyle\Delta\geq \max_{v\in V(G)}\{d^-(v)\}$. Let $v_0$ be a vertex with $d^+(v_0)=\Delta$ and denote the out-neighbourhood of $v_0$ by $N_0:=N^+_G(v_0)$. Let us assume that $v_1$ is a vertex for which $N_1:=N^+(v_1)- N_0$ is largest possible. 
		
		Note that by $2$-geodecity, for each fixed first vertex and last arc we have at most one path of length three; similarly we have at most one path of length three for each fixed first arc and last vertex.  We will use make of this several times in the following argument.
		
		There are at most $2n$ arcs starting from $N_0\cup N_1$ since $G$ is $2$-geodetic. Hence, by the observation of the previous paragraph, the number of 3-paths starting from $N_0$ or $N_1$ has quadratic order. %Since there are linear number of choices for the first edge and linear number of choices for the last vertex and since $G$ is $2$-geodetic for each fixed first edge and the last vertex there are at most one path of length three. 
		Similarly there are at most a quadratic number of paths of length three with third vertex lying in $N_0$ or $N_1$. Therefore since the desired upper bound for the number of 3-paths is cubic in order, we may ignore paths of length three which contain a vertex from $N_0$ or $N_1$ as the first or third vertex.
		
		Let us denote the number of vertices in $N_1$ by $x$. Since there are at most $(n-\Delta-x)$ choices for the first or the third vertex and at most $\Delta$ choices for the last vertex after fixing the third vertex, there are at most $(n-\Delta-x)^2\Delta + O(n^2)$ directed paths of length three. Using elementary calculus it is simple to check that we have
		\[
		(n-\Delta-x)^2\Delta\leq \frac{n^3}{8}
		\]
		if $\Delta+x \geq \frac{n}{2}$ or $\Delta\leq \frac{3-\sqrt{5}}{4}n$. Hence if $\Delta+x \geq \frac{n}{2}$ we have the desired inequality and we are done. 
		
		If $\Delta+x < \frac{n}{2}$ and $\Delta > \frac{3-\sqrt{5}}{4}n$, then we bound the number of paths of length three by a different function. The number of arcs in $G$ is at most $2n+\Delta^2+(n-\Delta-x)x$, where the $2n$ term bounds the number of arcs starting at $N_0\cup N_1$, the second term bounds the number of arcs entering $N_0$ and the third term bounds the number of arcs that are not incident from $N_0 \cup N_1$ or incident to $N_0$.  So by choosing the first vertex and the last arc and neglecting terms of order $O(n^2)$, the number of paths of length three is at most $f(x):=(n-\Delta-x)(\Delta^2+(n-\Delta-x)x)$. We have $f'(x)= 4 \Delta x - 2n \Delta + 3 x^2 - 4n x + n^2$ which is positive when $x=0$ and is negative when $x=\frac{n}{2}-\Delta$. Therefore the maximum of the function $f(x)$ in the interval $[0,\frac{n}{2}-\Delta]$ is attained at the smallest zero of $f'(x)$, $x_0= \frac{2n - 2\Delta - \sqrt{n^2 - 2n \Delta + 4 \Delta^2}}{3}$. Expanding and setting $\Delta = zn$ shows that the number of 3-paths minus $n^3/8$ is bounded above by
		
		\[ n^3\left(-\frac{11}{216}-\frac{2}{9}z+\frac{5}{9}z^2-\frac{11}{27}z^3+\frac{2}{27}(1-2z+4z^2)^{3/2}\right ).\]
		This function is negative in the interval $[\frac{3-\sqrt{5}}{4},\frac{1}{2}]$, completing the proof of the result. 
	\end{proof}
	
	\section*{Acknowledgements}
	The first author gratefully acknowledges funding support from EPSRC grant EP/W522338/1 and London Mathematical Society grant ECF-2021-27. The research of the third author is partially supported by the National Research, Development and Innovation Office NKFIH, grant K132696.

	%----------------------------------------------

	\section*{Statements and Declarations}
	Work of the first author was supported by EPSRC grant EP/W522338/1 and London Mathematical Society grant ECF-2021-27. Work of the third author was partially supported by the National Research, Development and Innovation Office NKFIH, grant K132696.
	
	The authors have no relevant financial or non-financial interests to disclose.
	
	The computational datasets referred to in the article are available from the corresponding author on request.
\end{document}